\documentclass[12pt]{amsart}
\usepackage{amsmath, amsthm, amssymb}
\usepackage[all]{xy}
\input epsf
\theoremstyle{plain}
\newtheorem{thm}[subsection]{Theorem}
\newtheorem{lem}[subsection]{Lemma}
\newtheorem{prop}[subsection]{Proposition}
\newtheorem{cor}[subsection]{Corollary}

\theoremstyle{definition}
\newtheorem{rem}[subsection]{Remark}
\newtheorem{defn}[subsection]{Definition}

\makeatletter

\newcommand{\Rmnum}[1]{\expandafter\@slowromancap\romannumeral #1@}
\makeatother

\begin{document}

\title{A Transversality theorem for some Classical Varieties}
\author {Chih-Chi Chou}
\address{University of Illinois at Chicago, Department of Mathematics, Statistics and Computer Science, Chicago, IL 60607\\ } \email{cchou20@uic.edu.}

\begin{abstract}
In 2009, de Fernex and Hacon  (\cite{T}) proposed a generalization of the notion of the singularities to normal varieties that are not 
$\mathbb{Q}\mathrm {-Gorenstein}$.   
Based on their work,  we generalize Kleiman's transversality theorem to subvarieties with log terminal or log canonical singularities.
We also show that some classical varieties, 
such as generic determinantal varieties, $W^r_d$ for general smooth curves, 
and certain Schubert varieties in $G(k, n)$ are log terminal.

\end{abstract}
\maketitle
%%%%%%%%%%%%%%%%%%%%%%%%

\setcounter{tocdepth}{1} \tableofcontents

\section{Introduction}
In higher dimensional birational geometry, one studies pair 
$(X, \Delta)$ consisting of a variety $X$ and a boundary divisor $\Delta$. 
The boundary, $\Delta$, plays several useful roles,  
such as  an effective divisor to make 
the canonical divisor $\mathbb{Q}$-Cartier, or a tool to apply adjunction formula to do induction.
However, in many cases there is not a canonical choice for $\Delta$.  
In \cite{T}, de Fernex and Hacon propose a new approach to birational geometry.
Instead of using  boundary divisor, they propose a more direct way to pull back non $\mathbb{Q}$-Cartier divisors.
They  define a relative canonical divisor which  generalizes the classical one, and then extend the singularity theory to non $\mathbb{Q} \mathrm{-Gorenstein}$ varieties.

They show that if $X$ is log terminal (resp. log canonical) in their sense, 
then there exists a boundary $\Delta$ such that $(X,\Delta)$ is Kawamata log terminal (resp. log canonical) in the old sense.
On the other hand, Kleiman's  theorem plays a central role in the study of intersection theory and enumerative geometry.
In this paper, based on de Fernex and Hacon's work we show the following theorem.
 
 \begin{thm} (Kleiman's Transversality Theorem)
 Let $X$ be a homogeneous variety with a group variety $G$ acting on it.
Let $Y, Z$ be subvarieties of $X$ such that $Z$ is smooth and $Y$ is log canonical (resp. log terminal). 
For any $g\in G$, denote $Y^g$ the transition of $Y$ by $g$.
Then there exists a non empty open $U\subset G$ such that
 $\forall g\in U$, $Y^g \times _X Z$ is log canonical (resp. log terminal). 
 \end{thm}
For the original version of Kleiman's Transversality theorem, see \cite{Kl} or \cite{Har} Theorem \Rmnum{3}.10.8.

There are some classical varieties that are well known to be normal but not $\mathbb{Q}\mathrm {-Gorenstein}$.  
To understand their singularities, one way is to find out a suitable boundary.
 However, for many of these varieties there is not a natural choice of the boundary divisor .
 Using the generalized notion of singularity, we can study the singularities in a more direct way.
 In this paper, we give a criterion for a normal variety to be log terminal.
 
 \begin{thm}
 Let $ f: Y \rightarrow X$ be a small resolution of a normal variety $X$, such that $Y$ is smooth and $-K_Y$ is relatively nef. Then $X$ is log terminal (in the sense of de Fernex and Hacon).
 Moreover, there is a boundary $\Delta$ such that $(X,\Delta)$ is canonical (in the old sense). 
 \end{thm} 
Here small means the exceptional locus is of codimension bigger than two.
Then we apply this theorem to some classical varieties such as generic determinatal varieties,
$W^r_d(C)$ for general smooth curve $C$, 
and certain Schubert varieties in $G(k, n)$.
For each of them, we construct a small resolution with a relatively nef anti-canonical divisor.
Then we conclude that they all have log terminal singularities.
Besides, there are boundary divisors such that these varieties have canonical singularities as pairs. 
In particular, they have rational singularities.

The organization of this paper is as follows. 
In the second section, we recall some notions from \cite{T} and also prove some propositions which will be used 
in the proofs of the main theorems.
The proofs of the main theorems are in section three.
Examples of log terminal varieties are given in the last three sections.

To simplify the notation, we will denote log canonical (resp. log terminal) by lc (resp. lt).
If we want to emphasize that it is the old notion, we will say lc (resp. lt) pair. 
Throughout the paper, the ground field is the complex numbers.

\subsection*{Acknowledgments}
The author would like to thank Izzet Coskun and Lawrence Ein for their help and suggestions, and his officemates Luigi Lombardi and Lei Song for their encouragement and stimulating conversations.

\section{Preliminaries}
In this section, we recall some definitions and prove some lemmas that will be used in the proofs of the main theorems.
First, we would like to generalize the notion of relative canonical divisor to normal varieties. The idea and notation will be mainly based on the paper \cite{T}. 

\begin{defn}
Let $ f: Y \rightarrow X $ be a morphism between two normal varieties.
Since $Y$ is normal, for any  prime divisor $E$  there is a valuation $V_E(\cdot)$ defined by the local ring of $E$.
For any ideal sheaf $\mathcal{I}$ on $Y$,  
the valuation $Val_E (\mathcal{I})$ is defined as
\begin{equation*}
Val_E (\mathcal{I}) :=min \{ V_E (\phi) \mid \phi \in \mathcal{I} (U), U \cap E \neq \emptyset \} 
\end{equation*}
For any Weil divisor $D$ on $X$, we  define the natural pull back as 
\begin{equation*}
f^{\natural} (D) := div(\mathcal{O}_X(-D) \cdot \mathcal{O}_Y ):=\sum _{E\subset Y} Val_E (\mathcal{O}_X(-D) \cdot \mathcal{O}_Y).
\end{equation*}
\end{defn}
One of the properties of the natural pull back is that
\begin{equation*}
\mathcal{O}_Y(-f^{\natural}D)=(\mathcal{O}_X(-D)\cdot \mathcal{O}_Y)^{\vee \vee}.
\end{equation*}

The natural pull back usually does not have the homogeneity property, i.e, $f^{\natural} (m K_X) \neq m  f^{\natural} (K_X)$ (See \cite{T} for example). As a result, instead of defining the relative canonical divisor directly, we have the following: 
\begin{defn}
The {\it {{m-th limiting relative canonical $\mathbb{Q}$ divisor}}} $K_{m, Y/X}$ is
\begin{equation*}
K_{m, Y/X}:= K_Y-\frac{1}{m}  f^{\natural} (mK_X).
\end{equation*}
Given any prime divisor $F$ on $Y$, we define the  {\it {{m-th limiting discrepancy}}} of $X$ to be 
\begin{equation*}
a_{m,F}(X):=\mathrm{ord}_F(K_{m, Y/X})
\end{equation*}
$X$ is said to be lc (resp. lt) if there is some positive integer $m_0$ such that $a_{m_0,F}(X) \ge -1$ $(resp. >-1)$ for every prime divisor $F$ over $X$. 

\end{defn}
 
 Let $g: V \rightarrow Y$ and $f: Y \rightarrow X$ be two birational morphisms, usually we do not expect 
 that $(fg)^{\natural}(mK_X)=g^{\natural}(f^{\natural}mK_X)$.
 However, the equality holds when $\mathcal{O}_X(-mK_X)\cdot \mathcal{O}_Y$ is an invertible sheaf (Lemma 2.7 in \cite{T}).  
A consequence of this property is the following lemma,
\begin{lem}(Lemma 3.5 in \cite{T})
Let m be a positive interger, and let $f:Y\rightarrow X$ be a proper birational morphism from a normal variety $Y$ such that $mK_Y$ is Cartier and $\mathcal{O}_X(-mK_X)\cdot \mathcal{O}_Y$ is invertible.
Then for every birational morphism $g:V\rightarrow Y$ we have 
\begin{equation*}
K_{m, V/X}=K_{m, V/Y}+ g^*K_{m, Y/X}
\end{equation*}
\end{lem} 

In particular, when study the singularities of a given variety $X$, it suffices to find a log resolution of $(X, \mathcal{O}_X(-mK_X))$.
More precisely, we just need to find a proper birational morphism $f:Y\rightarrow X$ such that $Y$ is smooth and $\mathcal{O}_X(-mK_X)\cdot \mathcal{O}_Y$ is the invertible sheaf of a divisor $F$, and the exceptional locus is a divisor $E$ such that $F \cup E$ is simple normal crossing.
The reason is that for any divisor $E$ over $X$, by taking a common resolution, we can assume $E$ is 
on a higher resolution $g:V\rightarrow Y$.
Then since $Y$ is smooth, the order of $K_{m, V/Y}$ over $E$ must be nonnegative, hence will not affect 
the type of the singularities of $X$.

Moreover,  for any $m \ge 2$, we can find an $m$-$compatible$ boundary $\Delta$ such that
$K_X+\Delta$ is $\mathbb{Q}$-Cartier and
\begin{equation}
K_{Y/X}^{\Delta}:=K_Y +\Delta _Y -f^*(K_X +\Delta) =K_{m,Y/X} 
\end{equation}
See \cite{T} for the definition of m-compatible boundary.
We will give a proof of a similar result in Lemma \ref{boundary}.

So if $X$ is log terminal (resp, log canonical) in the new notion, there is a boundary $\Delta$ (m-compatible for some positive integer m) such that $(X,\Delta)$ is Kawamata log terminal (resp, log canonical) in the classical sense.
On the other hand, for any boundary $\Delta$ such that $m(K_X+\Delta)$ is Cartier, we have $K_{Y/X}^{\Delta}\le K_{m,Y/X}$ (Remark 3.9 in \cite{T}).
As a result, we conclude that
\begin{prop}(Proposition 7.2 in \cite{T})
$X$ is log canonical (resp. log terminal) if and only if there is a boundary $\Delta$ such that $(X,\Delta)$ 
is a log canonical pair (resp. log terminal pair).
\end{prop} 

\begin{cor}
Let $X$ be a normal variety with log terminal singularities, then $X$ has rational singularities.
\end{cor} 

\begin{rem}
In \cite{T}, the authors also proposed a notion for being canonical and terminal,
however, there is an example which is canonical but not log terminal, see \cite{SU}. 
\end{rem}

Now we start to give some propositions which are not results in \cite{T}.
We will use these propositions in the next section.
These propositions are interesting by themselves and may be applied to other situations.

\begin{prop}\label{prop1}
Let  $f: X\rightarrow Y$ be a surjective, smooth morphism onto a normal variety $Y$  with log canonical (resp. log terminal) singularities, then $X$ is also a normal variety with log canonical (resp. log terminal) singularities.  
\end{prop} 

\begin{proof}
Take a log resolution $g:\tilde{Y} \rightarrow Y$ with smooth $\tilde{Y}$ and do base change as shown in the following
diagram. 
Since $f$ is smooth and $Y$ is normal, $X$ is also a normal variety(\cite{AS} Theorem 7.4.9).
Also $\tilde{f}$ is smooth morphism by base change, so we conclude that $\tilde{X}$ is a smooth variety and $\tilde{g}$ is a resolution of $X$.  

\centerline{ \xymatrix{  
 &\tilde{X}  \ar[d]_{\tilde{f}} \ar[r]^{\tilde{g}} 
& X \ar[d]_f	\\ & \tilde{Y} \ar[r]^g	& Y}}	

First note that $f$ and $\tilde{f}$ are flat, so it makes sense to pull back Weil divisor and the pull back is just the preimage of the divisor.
We claim that for any divisor $D$ on $Y$, we have the following commutative property,
\begin{equation}
\tilde{f}^*(g^{\natural}D)= \tilde{g}^{\natural}(f^*D). \label{qqq}
\end{equation}
Observe that $\mathcal{O}_Y(-D)\cdot \mathcal{O}_{\tilde{Y}} \subseteq \mathcal{O}(-g^{\natural}D)$,
and equality holds on codimension one subvarieties.
Since $\tilde{f}$ is smooth, pull back of any reduced divisor is still reduced (\cite{AS} Theorem 7.4.9).
 Ideals $(\mathcal{O}_Y(-D)\cdot \mathcal{O}_{\tilde{Y}})\cdot \mathcal{O}_{\tilde{X}} $ and $( \mathcal{O}(-g^{\natural}D))\cdot \mathcal{O}_{\tilde{X}}$ will determine exactly the same divisor on $\tilde{X}$, 
 which is the pull back of $g^{\natural}D$ by $\tilde{f}$.
 Also by commutativity of the diagram, $\tilde{g}^{\natural}(f^*D)$ is the divisor determined by the ideal 
$(\mathcal{O}_Y(-D)\cdot \mathcal{O}_{\tilde{Y}})\cdot \mathcal{O}_{\tilde{X}} $. 
As a result, we can conclude equation (\ref{qqq}).

For two different canonical divisors $K_{\tilde{Y}}$, $K'_{\tilde{Y}}$ on $\tilde{Y}$, $g_*(K_{\tilde{Y}})$ is linear equivalent to $g_*(K'_{\tilde{Y}})$.
So without loss of generality, we can fix $K_{\tilde{Y}}$ (resp. $K_{\tilde{X}}$) and assume $g_*(K_{\tilde{Y}})=K_Y$ (resp. $\tilde{g}_*(K_{\tilde{X}})=K_X$).
Also fix a divisor $E\sim _{lin}K_X-f^*K_Y$, note that $E$ is Cartier because $f$ is smooth.
Now we claim that 
\begin{equation}
\tilde{g}^{\natural}(mK_X)=m\tilde{g}^*(E)+\tilde{f}^*(g^{\natural}(mK_Y))
\end{equation}
This is because
\begin{eqnarray*}
\tilde{g}^{\natural}(mK_X)&=&\tilde{g}^{\natural}(m \tilde{g}_*K_{\tilde{X}})\\
&=&\tilde{g}^{\natural}(m \tilde{g}_*(K_{\tilde{X}}-\tilde{f}^*K_{\tilde{Y}}+\tilde{f}^*K_{\tilde{Y}}) ) \\
&=&\tilde{g}^{\natural}(m \tilde{g}_*(\tilde{g}^*(E)+\tilde{f}^*K_{\tilde{Y}}) )     \hspace{ 37pt } \mbox{ (by base change of dualizing sheaf) }\\
&=&m\tilde{g}^*(E)+ \tilde{g}^{\natural}(m\tilde{g}_*(\tilde{f}^*K_{\tilde{Y}}))  \hspace{ 30pt }   \mbox{(since $E$ is Cartier)}\\
&=&m\tilde{g}^*(E)+ \tilde{g}^{\natural}(mf^*(g_*K_{\tilde{Y}})) \hspace{ 30pt }   \mbox{(\cite{Fu} Proposition 1.7 in chapter one)}\\
&=&m\tilde{g}^*(E)+ \tilde{g}^{\natural}(mf^*(K_Y))\\
&=&m\tilde{g}^*(E)+\tilde{f}^*(g^{\natural}(mK_Y))\hspace{ 37pt } \mbox{ (by equation(\ref{qqq})) }
\end{eqnarray*}

Then we can calculate $K_{m, \tilde{X}/X}$ as following,
\begin{eqnarray*}
K_{m, \tilde{X}/X}&:=& K_{\tilde{X}}-\frac{1}{m}\cdot  \tilde{g}^{\natural} (mK_X)\\
&=&K_{\tilde{X}}-\frac{1}{m} \cdot(m\tilde{g}^*(E)+\tilde{f}^*(g^{\natural}(mK_Y)))\\
&=&K_{\tilde{X}}-\frac{1}{m} \cdot(m(K_{\tilde{X}}-\tilde{f}^*K_{\tilde{Y}})+\tilde{f}^*(g^{\natural}(mK_Y)))\\
&=&\tilde{f}^*(K_{\tilde{Y}}-\frac{1}{m}  f^{\natural} (mK_Y))=\tilde{f}^*(K_{m, \tilde{Y}/Y}) 
\end{eqnarray*}

So If we write $K_{m, \tilde{Y}/Y}= \sum a_i E_i$, where $E_i$'s are simple normal crossing exceptional 
divisors of $g$.
Then $K_{m, \tilde{X}/X}= \sum a_i E_i'$, where $E_i'$'s are the preimages of $E_i$'s.
Note that $E_i'$'s are also reduced since $\tilde{f}$ is smooth.
As a result, we conclude that $X$ is lc (resp. lt) is $Y$ is lc (resp. lt).

\end{proof}

\begin{prop} (Bertini)\label{prop2}
Let $X$ be a normal scheme with  lc (resp. lt) singularities.
Let $f$ be a morphism from $X$ to a projective variety $Y$. 
Then for a general point in $Y$, the fiber has lc (resp. lt) singularities.
\end{prop}

\begin{proof}
First we assume that $dim Y=1$. 
Take a very ample divisor $A$ on $Y$, consider the base point free linear system $|f^*A|$ on $X$. 
Note that $f^*A$ is an union of fibers and we claim that they have the same singularities as $X$ does.

By the result of Proposition 2.4, there is a $\mathbb{Q}$-Cartier divisor $\Delta$ such that 
$(X, \Delta)$ is a {\it{lc}} (resp. {\it {lt}}) pair.
Let $B$ be a general member in $|f^*A|$.
Then by  \cite{KO} Proposition 7.7,
\begin{equation*}
\mathrm{discrep}(B, \Delta |_B) \ge \mathrm{discrep}(X, \Delta). 
\end{equation*}
As a result, $B$ is {\it{lc}} (resp. {\it {lt}}) if $X$ is {\it{lc}} (resp. {\it {lt}}). 
Since $B$ is an union of fibers, we get the conclusion in this case.

For higher dimension $Y$, we do the same procedure as above.
Then replace $X$ by $B$, $Y$ by $f(B)$ with reduced scheme structure.
Note that $f(B)$ has dimension at most equal to $dimB$, and $B$ has codimension one in $X$.
So in each step, the dimension of $Y$ will drop by at least one.
Then by induction, the conclusion follows.

\end{proof} 

The following lemma, which will be used in the proof of the next proposition, is a slight modification of Theorem 5.4 in \cite{T}.
We include the proof with some modifications for the readers' convenience.  

\begin{lem} \label{boundary}
Let $X$ be a normal variety and $S$ a Cartier divisor on it.
Then for any $m\ge 2$, we can find a boundary $\Delta$ on $X$ such that $\Delta |_S$ is 
a m-compatible boundary of $S$.
\end{lem}

\begin{proof}
Take an effective divisor $D$ such that $K_x-D$ is Cartier, 
by adjunction formula, $(K_x +S-D)|_S=K_S-D_S$ is also Cartier.
Let $f:\tilde{X}\rightarrow X$ be a log resolution of $((X,S),\mathcal{O}(-mK_X)    )$,
and we assume  it is high enough that $f$ restricts to a resolution $\tilde{f}:\tilde{S}\rightarrow S$ as shown by the following diagram,

\centerline{ \xymatrix{  
 &\tilde{S}  \ar[d]_{\tilde{f}} \ar[r] 
& \tilde{X} \ar[d]^f	\\ & S \ar[r]	& X}}	
 
 Denote $\tilde{f}^{\natural}(mD_S)$ as $E_S$. 
 Since $K_S-D_S$ is Cartier, by Lemma 2.4 in \cite{T} and definition of natural pull back, we have
 \begin{eqnarray*}
 \tilde{f}^{\natural}(mK_S)
&=& \tilde{f}^{\natural}(m(K_S-D_S)+mD_S)\\
&=&\tilde{f}^*(m(K_S-D_S))+E_S
 \end{eqnarray*}
 As a result, we can write
 \begin{equation*}
 K_{m,  \tilde{S}/S}=K_{\tilde{S}}- \tilde{f}^*(K_S-D_S)-\frac{1}{m}E_S
 \end{equation*}
 
 On the other hand, 
 take a line bundle $L$ on $X$ positive enough such that  both $L\otimes \mathcal{O}_X(-mD)$ and its restriction to $S$, $L|_S\otimes \mathcal{O}_S(-mD_S)$ are globally generated.
 Let $M$ be general section of $L\otimes \mathcal{O}_X(-mD)$ such that,
 as an effective divisor, $M $ is reduced and has no common component with $D$ and $S$.
Define $\Delta _S$  as $\frac{1}{m} M|_S$.
Let $G=M+mD$ and denote its restriction to $S$ as $G_S:=M_S+mD_S$, 
then $K_S+\Delta_S=K_S-D_S+\frac{1}{m}G_S$ is $\mathbb{Q}$-Cartier. (Note that $G_S$ is Cartier.) 
We also have the following equalities:
 \begin{eqnarray*}
\tilde{f}^*G_S&=& \tilde{f}^{\natural}G_S  \hspace{ 60pt } \mbox{ (because $G_S$ is Cartier)}\\
&=&  \tilde{f}^{\natural}(M_S+mD_S)\\
&=& \tilde{f}^{-1}_*M_S+E_S\\
&=&m\tilde{f}^{-1}_*\Delta_S+E_S \hspace{ 12pt } \mbox{ (by definition)}
 \end{eqnarray*}
The last equation comes from the fact that we chose $M$ general, 
so we can assume $M$ does not pass through any locus we blow up.

Then we have
 \begin{eqnarray*}
 K_{\tilde{S}/S}^{\Delta _S}&:=& K_{\tilde{S}} + \tilde{f}^{-1}_*{\Delta _S}-\tilde{f}^*(K_S+\Delta_S)\\
 &=& K_{\tilde{S}} + \tilde{f}^{-1}_*{\Delta _S}-\tilde{f}^*(K_S+\Delta_S -\frac{1}{m}G_S)- \frac{1}{m}\tilde{f}^*G_S\\
 &=&K_{\tilde{S}}- \tilde{f}^*(K_S-D_S)-\frac{1}{m}E_S
 \end{eqnarray*}
 
 Hence $K_{m,  \tilde{S}/S}= K_{\tilde{S}/S}^{\Delta _S}$, i.e, 
 $\Delta_S=\Delta |_S$ is a m-compatible boundary of $S$.

\end{proof}

\begin{prop}\label{prop4}
Let $f: X\rightarrow Z$ be a morphism from a normal variety $X$ to a smooth variety $Z$. If every fiber of $f$is log canonical (resp. log terminal), then $X$ is log canonical (resp. log terminal).
\end{prop}

\begin{proof}
Given any point $x \in X$, let $z\in Z$ such that $x$ belongs to  the fiber of $z$. 
Assume dim$Z=1$, $z$ can be thought of as a Catier divisor and so is  $f^*z$.
Denote the Weil divisor corresponding to $f^*z$ as $S$, then by  lemma \ref{boundary},
 we can find a m-compatible boundary $\Delta_S$ such that $\Delta_S=\Delta |_S$ for some $\Delta$ on $X$. 
 By assumption, $(S, \Delta_S)$ is a lc (resp. klt) pair.
 Using  Inversion of adjunction Theorem (\cite{KM} Theorem 5.50), we conclude that $(X, \Delta +S)$ is a lc (resp. plt) pair in a neighborhood of $x$ .
By Proposition 2.4, this is equivalent to saying that $X$ is lc (resp. lt) near $x$.
 The argument applies to every point of $X$, so the proposition follows in this case.

 If $Z$ has higher dimension, say $n$, since $Z$ is smooth so every close point of $Z$ can be cut out by $n$ local coordinates.
 In particular, these local coordinates are Cartier divisors, so are their pull backs.
 As a result, we can apply Lemma \ref{boundary} and inversion of adjunction to this case.
 Then by induction on dimension of $Z$, the proposition follows.  
 
  \end{proof}

\begin{rem}
See \cite{Niu} for similar results, there Niu assumed some locally complete intersection conditions.
In \cite{BCMP}, a Bertini type theorem for rational singularities is proved.
\end{rem}

\section{Proofs of the Main Theorems}

\begin{proof} (Theorem 1.1)
Consider the morphism $\Phi : G\times Y\rightarrow X$ defined by the action of $G$ on $X$.
Given $x\in X$, the preimage $\Phi ^{-1} (x)$ by definition is
\begin{equation*}
\Phi ^{-1} (x)=\{
(g,y)\in G\times Y | g\cdot y =x
\}
\end{equation*} 
Now we consider the map $P: \Phi ^{-1} (x) \rightarrow Y$ induced by the projection from $G\times Y$ to $Y$ as shown in the following diagram.
Since $G$ acts transitively on X, $P$ is a surjective morphism.

\centerline{ \xymatrix{  
 &\Phi ^{-1}(x)  \ar[d]^P \ar[r]^{\Phi} 
& x \in X \\ & Y }}	

Fix $y\in Y$, we claim that the preimage $P^{-1}(y) \cong \mathrm{stab}_G(y)$.
Since $G$ acts transitively on $X$, so we can find $g'\in G$ such that $g' \cdot x=y$.
Then we have,
\begin{eqnarray*}
P^{-1}(y)&\cong & \{g\in G | g\cdot y =x         \}\\
               &\cong & \{g\in G | g'g\cdot y =y         \}\\
                &\cong & \mathrm{stab}_G(y).
\end{eqnarray*} 
Also note that $Y$ is a subvariety of $X$, and $G$ acts transitively on $Y$.
As a result,  every fiber of $P$ is isomorphic to $\mathrm{stab}_G(y)$ for some fixed $y$. 
By Theorem \Rmnum{3}. 9.9 and \Rmnum{3}.10.2 in  \cite{Har}, we conclude that $P$ is a smooth morphism.
Then by Proposition 2.7, $\Phi ^{-1}(x) $ is {lc} (resp. lt).

Next we consider the following diagram,

\centerline{ \xymatrix{  
 &W  \ar[d]^g \ar[r]^h 
& Z\ar[d]	\\ & G\times Y \ar[r]^{\Phi}	& X}}	
where $W=(G\times Y)\times _XZ$.
It is easy to see that every fiber of $\Phi$ is isomorphic to each other, so $\Phi$ is a flat morphism. 
By base change, $h$ is also flat.
This will imply $W$ is $S_2$ 
Also $Z$ is smooth, so the singular locus of $W$ lies in the fiber,
so $W$ is also $R_1$. 
By Serre condition, $W$ is normal.
So it makes sense to apply the notion of singularities here. 
By Proposition \ref{prop4}, $W$ is lc (resp. lt).  

Composite with $p: G\times Y\rightarrow G$, we have map from $W$ to $G$.
Apply Proposition \ref{prop2}, we can see that general fibers of the map from $W$ to $G$ are lc (resp. lt).
But for any $g\in G$, the fiber is just  $Y^g \times _X Z$, thus the proposition is proved.

\end{proof}

\begin{proof}(Theorem 1.2)
The goal is to prove that for some $m \gg 0$, we have the following isomorphism:
\begin{equation}\label{eq}
(\mathcal{O}_X(-mK_X) \cdot \mathcal{O}_Y )  \cong  \mathcal{O}_Y (-mK_Y).
\end{equation}
  So  the corresponding divisors are exactly the same, or equivalently, $K_{m, Y/X}=0$.
  In particular, $(\mathcal{O}_X(-mK_X) \cdot \mathcal{O}_Y ) $ is an invertible sheaf. 
  If $ f: Y \rightarrow X$ is already a log resolution, then we are done by Lemma 2.3 and the argument following it.
  If not, we do further blowing up. 
  But since $Y$ is smooth, blowing up will just make the relative canonical divisor even more effective. 
  In any case, $X$ is log terminal.  Showing equation (\ref{eq}) is a local problem, so from now on we assume $X$ is affine.

Let $U'$ and $U$ be open sets such that the restriction map $ f|_{U'}: U' \rightarrow U$ is an isomorphism. As a result, for any $m \ge 0$ we have the following equalities.
\begin{eqnarray*}
\Gamma (Y, \mathcal{O}_Y (-mK_Y)) &\cong& \Gamma (U', \mathcal{O}_Y (-mK_Y)|U')\\
& \cong& \Gamma (U,\mathcal{O}_X(-mK_X)|U)\\ 
&\cong& \Gamma (X, \mathcal{O}_X (-mK_X))
\end{eqnarray*}
The first and last equalities come from the fact that the complements of $U$and $U'$ are of codimension at least two. And the middle one comes from the isomorphism.
Since $X$ is affine, $ \mathcal{O}_X (-mK_X)$ is naturally generated by global sections for any $m \ge 0$. 
So if we can show that for some $m$, $ \mathcal{O}_Y (-mK_Y)$ is also generated by global sections, we can conclude equation $(1)$ since both these two sheaves have the same global sections as we proved. 
To this aim, we need the following lemma:
\begin{lem} (Relative Basepoint  Free Theorem, c.f. \cite{KM}) Let $Y$ be a smooth variety,  $f: Y \rightarrow X$ be a projective morphism. Let $D$ be an $f$-nef Cartier divisor and $aD-K_Y$ is $f$-nef and $f$-big for some $a\ge 0$, then $bD$ is $f$-free for some $b\gg 0$.
\end{lem} 

In our case, $-K_Y$ is the divisor $D$ in this lemma. It is $f$-nef by the condition in the proposition and $f$-big because $f$ is a birational morphism so the general fiber is of dimension zero (recall that a divisor $D$ being $f$-big means that the restriction of $D$ to general fiber is big ). As a result, by the lemma we have a surjective map $f^*f_* \mathcal{O} _Y(-bK_Y) \rightarrow \mathcal{O}_Y(-bK_Y)$ for some $b \gg 0$, or equivalently, there is a morphism $Y \rightarrow \mathbb{P} _X (f_* \mathcal{O}_Y(-bK_Y))$ determined by the surjective morphism. 
So we can conclude that the sheaf $\mathcal{O}_Y(-bK_Y)$ is generated by $\Gamma (X, f_* \mathcal{O} _Y (-bK_Y)) \cong \Gamma (Y,  \mathcal{O} _Y (-bK_Y))$. 
Thus the first part of Theorem1.2 is proved.

As shown by the above arguments, there is an integer $m>0$ such that $(\mathcal{O}_X(-mK_X) \cdot \mathcal{O}_Y ) $ is invertible and $K_{m,Y/X}=0$. 
Then the second part of Theorem 1.2 follows directly from Lemma 2.3 and equation (1) in Section 2.
\end{proof}

Theorem 1.2 also gives result on the singularities of image as following:
\begin{cor}
Given a birational morphism $f: Y \rightarrow X$, if the exceptional locus of $f$ is of codimension at least two and $-K_Y$ is relatively nef, then X has rational singularities.
\end{cor}
Along this line, see also \cite{FG}.

\section{Generic Determinatal Varieties }
In this section, we consider non zero linear maps  from a dimension $e$ vector space $E$ to a dimension $f$ vector space $F$. The collection of all these maps can be naturally represented by $X := \mathbb{P} ^{ef-1}$, each point corresponds to  a $f\times e$ matrix modulo scalar. 
Let  $X_k$ be the the locus on $X$ where the corresponding map is of rank at most $k$. It is well known that $X_k$ is normal, and Cohen-Macaulay if it is nonempty. 
In this section, we would like to prove moreover that it is log terminal.

First we would like to construct a resolution of $X_k$. Consider the following incidence correspondence, 
\begin{equation*}
Y_k := \{ (A,V) \in X \times {G}(e-k, E) | AV=0\},
\end{equation*}
where ${G}(e-k, E)$ is Grassmannian of $e-k$ subspace of $E$. To simplify notations, let ${G}(e-k, E)$ be denoted by 
$\mathbb{G}$ from now on. 
On $\mathbb{G}$ we have the tautological bundle sequence $0 \rightarrow S \rightarrow E \rightarrow Q \rightarrow 0$, where $S$ is the universal sub bundle  of rank $e-k$.   
Then the projection map   $p_2 :Y_k \rightarrow \mathbb{G}$ can be thought as the total space of projective bundle $\mathbb{P}(Q^*\otimes F)$ over $\mathbb{G}$, so $Y_k$ is smooth. 

Now we are going to prove $p_1:Y_k \rightarrow X_k$ is a small resolution.
First note that by the construction $p_1: Y_k \rightarrow X $ is a surjective map to $X_k$. 
Also since we are considering generic determintal varieties, $X_k$ is of codimension $(e-k)(f-k)$ in $X$, and the singular locus is $X_{k-1}$.
For any non-singular point $A\in X_k$, $A$ is of exact rank $k$, so the fiber in $Y_k$ is an unique point $(A,ker A)$.
Since the non-singular locus is an open dense set in $X_k$,  $p_1$ is a birational map.
We also know that $X_{k-1}$ is of codimension $e-k+f-k+1$ in $X_k$, and
over the generic point of $X_{k-1}$, the fiber is of dimension $e-k$.
So the exceptional locus is of codimension bigger than two, i.e., $p_1$ is a small resolution.

Next we need to show the anticanonical bundle $-K_{Y_k}$ on $Y_k$ is nef, to this aim, we consider the tangent bundle $T_{Y_k}$. First note that on $Y_k$ we also have tautology sequence $0 \rightarrow U \rightarrow Q^*\otimes F\rightarrow (Q^*\otimes F)/U \rightarrow 0$, $U$ is of rank one. (Here we ignore all the pull back symbol to simplify the expression). Then we can write the first Chern class of $T_{Y_k}$ as
\begin{eqnarray*} 
c_1(-K_{Y_k})=c_1(T_{Y_k})&=&c_1(T_{\mathbb{G}})+c_1(T_{Y_k/\mathbb{G}})\\
&=& c_1(S^* \otimes Q) + c_1(U^*\otimes (Q^*\otimes F)/U)\\
&=& e\cdot \sigma _1+(kf-1)\cdot \sigma _2-f\cdot \sigma _1 +\sigma _2\\
&=& (e-f) \cdot \sigma _1 + kf\cdot \sigma _2
\end{eqnarray*} 
where $ \sigma _1$ (resp,  $\sigma _2$) denote the first Chern class of $S^*$ (resp, $U^*$).
So we have $-K_{Y_k}$ is nef if $e \ge f$ and by Theorem 1.2 $X_k$ is log terminal in this case. 

On the other hand, if $e\le f$ we consider the following resolution,
\begin{equation*}
Z_k := \{ (A,V) \in X \times {G}(k, F) | Im (A) \subseteq V\},
\end{equation*}
Over $\mathbb{G}:= G(k, F)$, there is a universal sequence $0 \rightarrow S \rightarrow F$ where $S$ is of rank $k$. Then $Z_k$ can be seen as total space of the bundle $\mathbb{P}(E^* \otimes S)\rightarrow \mathbb{G}$. 
Following the same argument as above, $p_2 : Z_k\rightarrow X_k$ is a small resolution. Moreover, still denote the tautology bundle on $\mathbb{P}(E^* \otimes S)$ by $U$, we have 
\begin{eqnarray*} 
c_1(-K_{Z_k})=c_1(T_{Z_k})&=&c_1(T_{\mathbb{G}})+c_1(T_{Z_k/\mathbb{G}})\\
&=& c_1(S^* \otimes Q) + c_1(U^*\otimes (E^*\otimes S)/U)\\
&=& f\cdot \sigma _1+(ke-1)\cdot \sigma _2-e\cdot \sigma _1 +\sigma _2\\
&=& (f-e) \cdot \sigma _1 + ke\cdot \sigma _2
\end{eqnarray*} 
is nef again. So we prove in the case $e\le f$, $Z_k$ is a resolution of $X_k$ with $-K_{Z_k}$ being nef.  
In conclusion, by Theorem 1.2, the generic determinantal varieties are log terminal, and if we chose suitable boundary divisors they are canonical as pairs.    

\begin{rem}
In this section, we constructed two resolution $Y_k$ and $Z_k$, which in fact are flips to each other over $X_k$ if we fix $E$ and $F$. 
The idea is inspired from \cite{Tot}, where a similar construction was used to give a counterexample of the rigidity of the nef cone.

The result that the generic determinantal varieties being log terminal was first proved by Israel Vainsencher in \cite{IV},  see also section 3 in \cite{JS} for explicit calculation of log discrepancies.

To understand generic determinantal  varieties from the view point of arc space, see \cite{Roi}

After finishing this paper, the author is informed that J-C Hsiao \cite{Hsiao} also got the same result independently by different method.
\end{rem}

%%%%%%%%%%%%%%%%%%%%%%%%%%%%%%%%%%%%%%%%%%%%%%%%%%%%%%%%
%%%%%%%%%%%%%%%%%%%%%%%%%%%%%%%%%%%%%%%%%%%%%%%%%%%%%%%%

\section{$W^r_d$ in $\mathbf{Pic}^d (C)$}
We first recall the definition and construction of $W^r_d$.
\begin{equation*}
Supp(W^r_d)=\{L \in \mathbf{Pic}^d (C) | h^0(L) \geq r+1                \}
\end{equation*} 
To give $W^r_d$ a scheme structure we can think it as a determinantal variety in $\mathbf{Pic}^d (C)$ in the following way: Consider projection map
\begin{equation*}
v: C \times  \mathbf{Pic}^d (C) \rightarrow \mathbf{Pic}^d (C)
\end{equation*}
Fix a Poincare bundle $\mathcal{L}$ on $C \times \mathbf{Pic}^d (C)$ and take an effective divisor $E$ positive enough so that  $H^1(C,\mathcal{L} |_C \otimes E)=0$, or equivalently, $R^1v_* \mathcal{L}(\Gamma)=0$ where $\Gamma $ is defined as $E \times \mathbf{Pic}^d (C)$. Then we have an induced map $\gamma : v_* \mathcal{L}(\Gamma) \rightarrow v_* \Gamma $ between two locally free sheaves of rank $d+m+1-g$ and $m$ respectively. $W^r_d$ can be defined as the degeneracy locus of $rk(\gamma) \le m+d-g-r$ in  $\mathbf{Pic}^d (C)$.  It is of expected dimension $\rho = g-(r+1)(g-d+r)$ and is nonempty if $\rho \ge 0$ (\cite{ACGH}).

Denote $v_* \mathcal{L}(\Gamma)$ by $K^0$ and $v_* \Gamma$ by $K^1$. There is a natural resolution of $W^r_d$ constructed as the following: Consider the Grassmann bundle $\pi : G(r+1, K^0) \rightarrow \mathbf{Pic}^d (C)$, with an universal sequence
\begin{equation*}
0\rightarrow S \rightarrow \pi ^* K^0 \rightarrow Q \rightarrow 0
\end{equation*} 
then the resolution of $W^r_d$ is defined as $G^r_d(C):=Zero(S^* \otimes \pi ^* K^1)$. It is smooth by the following theorem,

\begin{thm} (Smoothness Theorem \cite{ACGH} V.1.6)
Let C be a general curve of genus g. Let d,r be integers such that $d\ge 1$ and $r\ge 0$. Then $G^r_d (C)$ is smooth of dimension $\rho$.
\end{thm}

Since $G^r_d$ is smooth, we have the exact sequence 
\begin{equation*}
0\rightarrow T_{G^r_d} \rightarrow T_{Gr}|_{G^r_d} \rightarrow N_{G^r_d /Gr} \rightarrow 0
\end{equation*}
where $Gr$ is the Grassmann bundle $ G(r+1, K^0) \rightarrow \mathbf{Pic}^d (C)$ and $N_{G^r_d /Gr}$ in this case is $S^* \otimes \pi ^* K^1$. So this exact sequence can be translated to 
\begin{equation*}
0\rightarrow T_{G^r_d} \rightarrow (S^* \otimes \pi ^*K^0/S) | _{G^r_d}\rightarrow S^* \otimes \pi ^* K^1 \rightarrow 0
\end{equation*}
As a result, $c_1(T_{G^r_d} )=(d+1-g)\cdot c_1(S^*)$, i.e., $T_{G^r_d}$ is relatively nef if $d \ge g-1$. 
However, by the natural correspondence $L\rightarrow K_c \otimes L^\vee$ we have isomorphism
$W^r_d \cong W_{2g-2-d}^{g-d+r-1}$.
The $d \le g-1$ case is covered by this isomorphism, so it suffices to assume $d \ge g-1$.

By Riemann-Roch theorem, for $r \le d-g$ we have $W^r_d(C)=\mathbf{Pic}^d (C)$, which is smooth
, so we can limit ourself to $r > d-g$.
In this case, the singular locus of $W^r_d(C)$ is $W^{r+1}_d(C)$, and the dimension will be:
\begin{eqnarray*}
dim(W^{r+1}_d(C))&=&g-(r+2)(r+1-d+g)\\
&=&\rho -(r+1)-(r-d+g)-1
\end{eqnarray*}
However the general fiber of $\pi$ restricted to $W^{r+1}_d(C)$ is of dimension $r+1$, so the exceptional locus of $\pi$ is of codimension at least two, i.e., $\pi$ is a small resolution.
Combining the result and Theorem 1.2,  we conclude that 
\begin{prop}
let C be a general smooth  curve,  then $W^r_d(C)$ is log terminal. In particular, it has rational singularities.
\end{prop}

%%%%%%%%%%%%%%%%%%%%%%%%%%%%%%%%%%%%%%%%%%%%%%%%%%%%%%%%
%%%%%%%%%%%%%%%%%%%%%%%%%%%%%%%%%%%%%%%%%%%%%%%%%%%%%%%%

\section{Schubert Varieties in $G(k,n)$}

We construct an iterated tower of Grassmann bundles to resolve the singularities, see \cite{I} and \cite{Z}.  
Under some conditions on the canonical divisor, we can show that this resolution satisfies the requirements in Theorem 1.2.
After finishing this paper, the author was informed that  D. Anderson and A. Stapledon proved that all Schubert varieties are log terminal. Since their result is more general we omit the proof and refer the interested readers to  \cite{DAS}. 

With their result in \cite{DAS} and Theorem 1.1, we have an interesting corollary.
Before the statement of the corollary we first recall the definition of intersection variety, see \cite{Br} and \cite{IB}.

\begin{defn}
Given Schubert varieties $X_{u_1}, \cdots , X_{u_r}\in G(k, n)$, and let $g_{\bullet}=(g_1,\cdots , g_r)$
be a general r-tuple of elements in $G^r$, where $G$ is $GL_n$.
The intersection variety $R(u_1,\cdots,u_r; g_{\bullet})$ is defined as the intersection of the translated Schubert varieties $g_iX_{u_i}$ :
\begin{equation*}
R(u_1,\cdots,u_r; g_{\bullet})=g_1X_{u_1} \cap\cdots \cap g_rX_{u_r}.
\end{equation*}
\end{defn}

\begin{cor}(Suggested by Izzet Coskun)
If the singular locus of $X_{u_1}, \cdots , X_{u_r}$ are disjoint with each other, then the intersection 
variety $R(u_1,\cdots,u_r; g_{\bullet})$ is log terminal.
\end{cor}

\begin{proof}
It suffices to prove the case where $r=2$, which is also called {\it{Richardson Variety}}.
Consider $g_1X_{u_1}$ and $g_2X_{u_2}$ , their singular locus are disjoint since $g_1$ and $g_2$ are general.
Let $U$ be the smooth locus of $g_2X_{u_2}$, then $g_1X_{u_1} \cap U$ is log terminal by Theorem 1.1.
 The same is true for $U^c\cap g_1X_{u_1}=U^c\cap V$, where $V$ is the smooth locus of $g_1X_{u_1}$.
\end{proof}

\end{document}